\documentclass{amsart}
\usepackage{graphicx}
\usepackage{latexsym}
\usepackage{amsfonts}
\usepackage[all]{xy}
\usepackage{amssymb, mathrsfs, amsfonts, amsmath}
\usepackage{amsbsy}
\usepackage{amsfonts}
\newtheorem{corollary}{Corollary}

\newtheorem{lemma}{Lemma}
\newtheorem{proposition}{Proposition}
\newtheorem{remark}{Remark}
\newtheorem{theorem}{Theorem}
\newtheorem{example}{Example}
\numberwithin{equation}{section}

\begin{document}
	\title[constant flag curvature]{On spherically symmetric Finsler metric with scalar and constant flag curvature}
	\author{Newton Sol\'orzano$^{1}$ and Benedito Leandro$^{2}$ }
	
	\address{Universidade Federal da Integra\c{c}\~{a}o Latino-Americana  |  Avenida Silvio Am\'erico Sasdelli, 1842 - Vila A, Edif\'icio Comercial Lorivo | CEP: 85866-000 | Caixa Postal 2044 - Foz do Igua\c{c}u - Paran\'a.}
	\email{newton.chavez@unila.edu.br$^{1}$}
	\address{Universidade de Federal de Goi\'as- UFG, IME, 74690-900, Goi\^ania- GO, Brazil.}
	\email{bleandroneto@ufg.br$^{2}$}
	

	\begin{abstract}
		In this paper we study spherically symmetric metrics on a symmetric space in $\mathbb{R}^n$ with scalar and constant flag curvature and we also obtain families of this type of metrics.  Many explicit examples are provided for Douglas metrics with scalar and constant flag curvature.  Furthermore, new examples of projectively flat Finsler metrics are given. We also provide a family of spherically symmetric Finsler metrics which are not of Douglas type.
	\end{abstract}
	
	\keywords{Finsler metric, Douglas curvature, Flag curvature, spherically symmetric.}   
	\subjclass[2010]{53B40, 53C60, 58B20.}
	\date{\today}

	\maketitle
	\section{Introduction}
	Finsler metrics are more general than Riemannian metrics since the tangent norms do not need to be induced by inner products. Then it is natural to study Riemannian properties in Finsler geometry. The Riemannian sectional curvature in Finsler geometry is given by the flag curvature $K(P,y)$ defined by (cf. \cite{CS2}): \[K(P,y):=\frac{g_y(R_y(u),u)}{g_y(y,y)g_y(u,u)-[g_y(y,u)]^2},\] for a tangent plane $P\subset T_xM $ containing $y$, and $u\in P$ is such that $P=span\{y,u\}$.

	A Finsler metric is said to be of scalar flag curvature if $K(P,y)=K(x,y)$ and is said to be constant flag curvature if the flag curvature $K(P,y)=K$ is constant. 
	Many researchers are interested in classifying and look for Finsler metrics with scalar and constant flag curvature. 
	
	Till now, some progress was made. In \cite{bao2004}, the authors provided a local classification theorem for Randers metrics, while in \cite{LZhou2010} and \cite{ZSheenCivi20018} they studied some class of $(\alpha,\beta)$-metrics with constant flag curvature, and in \cite{Xia2017} the author obtained some non-projectively flat Finsler metrics of scalar (resp. constant) flag curvature for some class of general $(\alpha, \beta)-$metrics.

	The Hilbert's Fourth problem claims to classify metrics with the property that their geodesics are straight lines. In Finsler Geometry this problem is equivalent to looking for projectivelly flat Finsler Metrics in an open domain $U\subset R^n.$  A Finsler metric $F$ is called Locally Projectively flat if at any point $p\in U$, there is a local coordinate system $(x^i)$ in which $F$ is projectively flat.
	
	For Riemannian metrics, Hilbert's fourth problem is solved by the Beltrami's theorem, which states that a Riemannian metric is locally projectively flat if and only if it is of constant sectional curvature. It is well known that every locally projectively flat Finsler metric has scalar flag curvature. However, the opposite is not necessarily true. There are Finsler metrics of scalar flag curvature, which are not locally projectively flat (cf. \cite{bao2004}). In \cite{Xia2017}, the author showed that a class of general $(\alpha, \beta)-$metric has scalar flag curvature if and only if it is locally projectively flat.
	
	The Douglas curvature introduced by J. Douglas in 1927 \cite{D} is very important in Finsler geometry because it is a projectively invariant. Namely, if two Finsler metrics $F$ and $\overline{F}$ are projectively equivalent, then $F$ and $\overline{F}$ have the same Douglas curvature.  Finsler metrics with vanishing Douglas curvature are called Douglas metrics. The class of Douglas metrics contains all Riemannian metrics and the locally projectively flat Finsler metrics.  However, there are many Douglas metrics which are not Riemannian. There are also many Douglas metrics which are not locally projectively flat \cite{MY}. 
	
	
	There are many important examples of Finsler metrics such as Funk Metric, Bryant's metrics, Chern-Shen's metric, which satisfy \[F(Ax,Ay)=F(x,y)\]
	for all $A\in O(n).$  This type of metrics is called spherically symmetric and was first studied by Rutz \cite{Rutz1996}. It can be expressed by (cf. \cite{HM1,HM2,Z}):
	\begin{align}\label{finmet}
	F(x,y)=|y|\phi(|x|,\dfrac{\langle x,y\rangle}{|y|}), 
	\end{align}
	where $\phi(r,s)$ is a positive real function such that
	\begin{align*}\label{defrs}
	r:=|x|,\quad s:=\frac{\langle
		x,\,y\rangle}{|y|}.
	\end{align*}
	Here, $|\cdot|$ and $\langle\cdot\,,\,\cdot\rangle$ are the Euclidean norm and inner product respectively. Note that the spherically symmetric Finsler metrics are a class of the general $(\alpha, \beta)-$metrics.
	
	The geodesic coefficients of spherically symmetric Finsler metric are given by (cf. \cite{HM1,YZ}):
	\[G^i=|y|Py^i + |y|^2Qx^i,\]
	where 
	\begin{align}\label{defQ}
	Q:=\frac
	1{2r}\frac{r\phi_{ss}-\phi_r+s\phi_{rs}}{\phi-s\phi_s+(r^2-s^2)\phi_{ss}},\quad
	\quad r:=|x|,\quad s:=\frac{\langle
		x,\,y\rangle}{|y|}
	\end{align}
	and
	\begin{align}\label{defP}
	P:=\frac{r\phi_s+s\phi_r}{2r\phi}-\frac
	Q{\phi}\left[s\phi+(r^2-s^2)\phi_s\right].
	\end{align}
	
	It is important to point out that Bryant et al. \cite{Bryant2017} provided a complete study of  spherically symmetric Finsler surfaces of constant flag curvature. Moreover, in \cite{SEZShen2015} the authors improved the results obtained in \cite{2012arXiv1202.4543M} for metrics in $\mathbb{R}^n$, and then they proved some new examples of spherically symmetric Finsler metrics with constant flag curvature. 
	
	To state our results, we introduce the following notation.
	For a positive smooth function $\phi=\phi(r,s),$  let 
	\begin{align}\label{R2}
	R_1&:= P^2-\frac{1}{r}(sP_r+rP_s) + 2Q[1+sP+(r^2-s^2)P_s],\nonumber\\
	R_2&:= 2Q(2Q-sQ_s)+\frac{1}{r}(2Q_r-sQ_{rs}-rQ_{ss})+(r^2-s^2)(2QQ_{ss}-Q_s^2),\\
	R_3&:=\frac{1}{r}\left(\left[(r^2-s^2)2Q-1\right]rP_{ss}-sP_{rs}+P_r+r2Q(P-sP_s)\right),\nonumber
	\end{align}
	where $Q=Q(r,s)$ and $P=P(r,s)$ are given by (\ref{defQ}) and (\ref{defP}), respectively.

	In this paper we study  the system of partial differential equations obtained in \cite{SEZShen2015}, \cite{2012arXiv1202.4543M} that describes spherically symmetric Finsler metrics with scalar and constant flag curvature. Contrasting to results obtained in \cite{MANA:MANA201400124,2012arXiv1202.4543M,HM1}, which studied Douglas metrics and obtained examples, we provide a family of functions that can define Finsler metrics with scalar flag curvature, but not of Douglas type. Furthermore, we obtain new examples of Douglas metrics with scalar curvature $K(x,y)$ and consequently examples of projectively flat Finsler metrics (see Section \ref{section4} right down below).

	Also, the classification of Douglas metrics with constant flag curvature (or locally projectively flat) that we present suggests a method able to build new examples (see Remark before the Theorem \ref{theo4}).

	\vspace{0.5cm}

	\
	In Section 3 we prove a classification for some families of Finsler metrics (not Douglas) with scalar flag curvature (see Theorem \ref{theonot}). Moreover, we provide a classification for the Douglas metrics with scalar flag curvature (cf. Theorem \ref{teo1}) and with constant flag curvature in Theorem \ref{theo3}. It is important to emphasize that we use the characteristic method to prove our main results.
	
	\
	Let us establish some additional notation before announcing the main results. For some $g=g(r)$ differentiable function, we define:
	\begin{align}\label{T}
	T(r):=\frac{e^{-\int\frac{4rg(r)}{1-2r^2g(r)}dr}}{(1-2r^2g(r))^2} \quad \text{ and } \quad \overline{T}(r)=4\int\frac{g'(r)+2rg^2(r)}{1-2r^3g(r)}T(r)dr.
	\end{align}
	Here, $M_s^n$ denotes one of the following symmetric subsets of $\mathbb{R}^n$: $\mathbb{B}(\nu)=\left\{x\in \mathbb{R}^n\,; \,|x|<\nu\right\},$  $\mathbb{B}(\nu_1) \hspace{0.5mm}\backslash \hspace{0.5mm}\mathbb{B}(\nu_2)$ or $\mathbb{R}^n,$ where $0<\nu<+\infty$ and $0\leq\nu_2<\nu_{1}$. 
	
	Our first theorem provides a family of Finsler metrics (non Douglas) which are of scalar flag curvature.
	
	\begin{theorem}\label{theonot}
		A spherically symmetric Finsler metric satisfying \eqref{finmet} in which 
		\begin{align}\label{phinondouglas}
		\phi(r,s)=sh(r)-s\int \frac{\eta(\overline{\varphi}(r,\,s))}{s^2\sqrt{r^2-s^2}}ds,
		\end{align} 
		where $\eta,\,h$ are arbitrary differential real functions, $k$ constant and 
		\begin{align}\label{varphi2}
		\overline{\varphi}(r,s)=\frac{r^{2(k^2+1)}(r^2-s^2)e^{2k\arctan\left({k-\frac{(1+k^2)s}{\sqrt{r^2-s^2}}}\right)}}{k^2s^2-2ks\sqrt{r^2-s^2}+r^2},
		\end{align}  is of scalar flag curvature. Moreover, when $k=0$, then $F=|y|\phi(r,s)$ is projectively flat.
	\end{theorem}

	\begin{theorem}\label{teo1}
		Any spherically symmetric Douglas metric with scalar flag curvature on $M_s^n$ ($n\geq 3$) is given by \eqref{finmet},
		where 
		\begin{align}\label{phiscalar}
		\phi(r,s)=s\left(h(r)-\int \frac{\eta(\varphi(r,s))}{s^2\sqrt{r^2-s^2}}ds\right)\quad\mbox{and}\quad\varphi(r,s)=\frac{r^2-s^2}{(r^2-s^2)\overline{T}(r) -T(r)},
		\end{align}
		with $\phi$ satisfying \eqref{R2}. Here $h$ and $\eta$ are arbitrary differentiable real functions such that
		\begin{align}\label{etateo2}
		\frac{-\sqrt{r^2-s^2}}{s}\frac{\partial \eta}{\partial s}>0\quad\mbox{and}\quad 
		\frac{\eta}{\sqrt{r^2-s^2}}> 0.
		\end{align}
	\end{theorem}

	\
	
	In \cite{2012arXiv1202.4543M} Mo and Zhou proved that a Finsler metric satisfying  \eqref{finmet} is of constant flag curvature  $K$ if, and only if, $R_1=K\phi^2$ and $R_2=R_3=0$. Then, Sevim et al. developed these result (see Theorem 4), and showed that these three equations can be reduced to $R_2=R_3=0$ (cf. \cite{SEZShen2015}).
	
	\
	
	Let $g(r),$ $\overline{h}(r)$ and $\overline{\eta}(\varphi)$ be smooth real functions, considering $R_2=R_3=0$ we can express $P=P(r,s)$ and $Q=Q(r,s)$ by (see the proof of Theorem \ref{theo3}):
	\begin{align}\label{Hsh}
	Q(r,s)=g(r)+\frac{g'(r)+2rg^2(r)}{r-2r^3g(r)}s^2\quad\mbox{and}\quad
	P(r,s) =s \left( \overline{h}(r) - \int{\frac{\overline{\eta}(\varphi(r,s))}{s^2\sqrt{r^2-s^2}} ds}\bigr.\right),
	\end{align} where 
	\begin{align*}
	\varphi(r,s)=\frac{r^2-s^2}{(r^2-s^2)\overline{T}(r)-T(r)}.
	\end{align*}
	Here $T(r)$ and $\overline{T}(r)$ are given by (\ref{T}).
	\begin{remark}\label{remark01}
		Note that $P$ given by \eqref{Hsh} is equal to $\phi$ in (\ref{phiscalar}) because these two functions must satisfy the same equation (see the proof of Theorem \ref{theo3}).
	\end{remark}
	
	Moreover, we define $U=U(r,s)$ by 
	\begin{align}\label{conditionPQ}
	U(r,s)=&\frac{(r^2-s^2)(sP_s-2P)-s(1+sP)}{(r^2-s^2)(2Q-sQ_s)-sP-1}.
	\end{align}
	This function satisfies the following identity: 
	\begin{eqnarray}\label{condU}
	s\left[sU_r+(1-(r^2-s^2)2Q)rU_s\right]&=&r\left[1+(r^2-s^2)2(sQ_s-Q)\right]U\nonumber\\
	&&+2r(r^2-s^2)(sP_s-P).
	\end{eqnarray} 
	
	Inspired by \cite{SEZShen2015} and \cite{2012arXiv1202.4543M}, we get the next result, which ensures a new method to obtain the results provided in \cite{Xia2017}.

	\begin{theorem}\label{theo3}
		Any spherically symmetric Douglas metric with constant flag curvature on $M_s^n$ ($n\geq 3$) is given by 	\eqref{finmet}, where
		\begin{align}\label{phiP}
		\phi(r,s)=&\sqrt{r^2-s^2} \operatorname{exp}\left({\int\frac{U(r,s)}{r^2-s^2}ds}\right)
		\end{align}
		such that $U=U(r,s)$ satisfies the condition (\ref{condU}),  
		\begin{align}\label{fins1}
		\phi(r,s)-s\phi_s(r,s)>0 
		\quad\mbox{and}\quad
		\phi(r,s)-s\phi_s(r,s) + (r^2-s^2)\phi_{ss}(r,s)> 0.
		\end{align}
	\end{theorem}

	\
	
	Theorem \ref{theo3} gives us a method to provide examples of Finsler metrics with constant flag curvature: 1) Designate $\overline{h}(r),$ $g(r)$ and $\overline{\eta}(\varphi),$ and establish $P$ and $Q$ by \eqref{Hsh}. 2) Prove that (\ref{condU}) holds (or provides conditions for the existence of the three functions in step 1 using  (\ref{condU})). 3) Define $\phi(r,s)$ by
	\eqref{phiP}
	and then verify that \eqref{finmet} is a Finsler metric (it is an example of Douglas metric with constant flag curvature). 4) Make sure that $P$ and $\phi$ satisfy the same equation (see Remark \ref{remark01}), then define $P(r,s):=\phi(r,s)$ , where $\phi(r,s)$ is given in step 3. 5) Repeat steps 2-4. We will use this to construct an example further on (see Example \ref{ex10}).

	The following theorem prove the curvature for the metric \eqref{finmet} presented in Theorem \ref{theo3}.
	\begin{theorem}\label{theo4}
		The curvature $K$ of the metric given in Theorem \ref{theo3} is given by 
		\[K=\frac{R_1}{\phi^2}.\]
	\end{theorem}
	%

	Note that if $K\neq 0$, by Theorem 2, Theorem 4 and Lemma 5 (see below), we have another way to calculate $\phi(r,s)$.
	\begin{corollary}
		Let $F=|y|\phi(r,s)$ be a spherically symmetric Douglas metric on $M_s^n\subset \mathbb{R}^n$ ($n\geq 3$) with constant flag curvature $K\neq 0$, then \[K\phi^2=P^2-\frac{1}{r}(sP_r+rP_s) + 2Q[1+sP+(r^2-s^2)P_s]\]
		where $P=P(r,s)$ and $Q=Q(r,s)$ are given by Theorem \ref{theo3}.
	\end{corollary}

	\section{Background}

	Let $M$ be a manifold and let $TM=\cup_{x\in M}T_xM$ be the tangent
	bundle of $M$, where $T_xM$ is the tangent space at $x\in M$. We
	set $TM_o:=TM\setminus\{0\}$ where $\{0\}$ stands for
	$\left\{(x,\,0)|\, x\in M,\, 0\in T_xM\right\}$. A {\em Finsler
		metric} on $M$ is a function $F:TM\to [0,\,\infty)$ with the
	following properties
	
	(a) $F$ is $C^{\infty}$ on $TM_o$;
	
	(b) At each point $x\in M$, the restriction $F_x:=F|_{T_xM}$ is a
	Minkowski norm on $T_xM$.

	\
	
	Finsler metrics defined on $M^n_s$ satisfying
	\begin{equation}
	F(Ax,\,Ay)=F(x,\,y) \label{1.4}
	\end{equation}
	for all $A\in O(n)$  are called {\em
		spherically symmetric}, or {\em orthogonally invariant} (cf. \cite{S}). Such metrics were first studied by Rutz in \cite{Rutz1996}. 
	
	\
	
	Before the demonstration of the main theorems, we collect some useful results. We start with a proposition first observed in \cite{MoSoTe1}.
	
	\begin{proposition}\label{propEdp}
		Let $f(r)$ and $g(r)$ be smooth functions on $I\subset \mathbb{R}$. The general solution of the transport equation \begin{align}\label{eqtransp}
		\left[1-(r^2-s^2)(2g(r)+f(r)s^2)\right]r{\psi}_s(r,s) + s{\psi}_r(r,s)=0,
		\end{align} 
		is given by
		\begin{align*}
		\psi(r,s)= \eta\left(\frac{{r^2-s^2}}{{(r^2-s^2)\int{2rfe^{\int{2r(2g+r^2f)dr}}dr-e^{\int{2r(2g+r^2f)dr}}}}}\right),
		\end{align*}
		where $\eta$ is any smooth function on $\mathbb{R}.$
	\end{proposition}

	\

	In \cite{D}, Douglas introduced the local functions 
	$D_j{}^i{}_{kl}$ on
	$\mathcal{T}\mathbb{B}^n(\nu)$ defined by
	\begin{equation}
	D_j{}^i{}_{kl}:=\frac{\partial^3}{\partial y^j\partial y^k\partial
		y^l}\left(G^i-\frac 1{n+1}\sum_m \frac{\partial G^m}{\partial
		y^m}y^i\right),  \label{3.1}
	\end{equation}
	in local coordinates $x^1,...,x^n$ and $y=\sum_i y^i \partial/\partial x^i$. 
	These functions are called  {\em Douglas curvature} \cite{D} and a Finsler metric 
	$F$ is said to be a {\em Douglas metric} if $D_j{}^i{}_{kl}=0$.
	In  \cite{MoSoTe1}, the authors proved that a spherically symmetric 
	Finsler metric on  $M_s^n$ is Douglas type if, and only if, there exist functions $f=f(r)$ and $g=g(r)$ such that,
	\begin{align}\label{1.5}
	Q=g+\frac{s^2}{2}f.
	\end{align}
	\noindent
	Note that $F=|y|\phi(r,s)$ is projectively flat if, and only if, $f=g=0.$

	\
	
	Let ${{R_j}^i}_{kl}$ denote the Riemannn curvature tensor of the Berwald connection and ${R^i}_k={{R_j}^i}_{kl}y^jy^l$. Now, it can be demonstrated that a Finsler metric $F$ on a manifold $M$ is of scalar flag curvature $K=K(x,y)$ if, and only if,
	\[{R^i}_ k=K\{F^2\delta^i_ k-g_{kl}y^ly^i\}.\]
	
	\begin{lemma}\cite{2012arXiv1202.4543M,MANA:MANA201400124}
		Let $F=|y|\phi\left(|x|,\,\frac{\langle
			x,\,y\rangle}{|y|}\right)$ be a spherically symmetric Finsler on $\Omega$, then the Riemann curvature tensor ${R^i}_j $ is given by
		\[{R^i}_j=R_1(|y|^2\delta^i_ j-y^iy^j)+|y|R_2(|y|x^j-sy^i)x^i+R_4(|y|x^j-sy^j)y^i,\]
		where $R_1,$ $R_2$ and  $R_3$ are given by \eqref{R2}, and  \[R_4:=\frac{1}{2}\left\{3R_3-[R_1]_s\right\}.\]
	\end{lemma}
	
	It is well known (cf. \cite{Mat1980}) that a Finsler metric has vanishing Weyl curvature if, and only if, it has scalar curvature. Using this result, the authors in \cite{MANA:MANA201400124} presented the next lemma.
	\begin{lemma}\label{lemmascalarK}\cite{MANA:MANA201400124}
		Let $m \geq 3$ and let $F = |y|\phi( |x| , \frac{<x,y>}
		{|y|} )$ be a spherically symmetric Finsler metric on $M_s^n$. Then $F$ is of scalar curvature if, and only if, $\phi(r,s)$ satisfies
		\begin{align}\label{R_2=0}
		R_2:=2Q(2Q-sQ_s)+\frac{1}{r}(2Q_r-sQ_{rs}-rQ_{ss})+(r^2-s^2)(2QQ_{ss}-Q_s^2)=0.
		\end{align}
	\end{lemma}

	\begin{lemma}\cite{2012arXiv1202.4543M}
		Let $F=|y|\phi\left(|x|,\,\frac{\langle
			x,\,y\rangle}{|y|}\right)$ be a spherically symmetric Finsler on $M_s^n$, then $F$ has constant flag curvature $K$ if, and only if,
		\begin{align}
		R_1&=K\phi^2 \label{condR1};\\
		R_2&=0 \label{CondR2};\\
		R_3&=0. \label{CondR3}
		\end{align}
	\end{lemma}
	
	\noindent The above result was improved in \cite{SEZShen2015}. Moreover, they also proved the theorem below.
	\begin{theorem}\cite{SEZShen2015}\label{theoersa}
		Let $F=|y|\phi(r,s)$ be a spherically symmetric Finsler metric on $M_s^n\subset \mathbb{R}^n$ ($n\geq 3$). Then $F$ is of constant flag curvature $K$ if, and only if, \[R_2=0, \quad R_3=0.\]
	\end{theorem}

	\section{Proof of Main Results}
	
	Here we use the characteristic method to obtain the general solution of the so called transport equation (cf. \cite{PR}) with  
	nonconstant coefficients 
	\begin{equation}\label{transpeq}
	\psi_r(r,s)+v(r,s)\psi_s(r,s)=P(r,s,\psi(r,s)). 
	\end{equation} 
	Note that \eqref{transpeq}
	can be written as $\vec{v} \cdot \nabla {\psi}=P$, where $\vec{v}=(1, v)$ and $\nabla {\psi}= ({\psi}_r , {\psi}_s)$. We can describe this equation as follows: we are seeking for a surface $z={\psi}(r,s)$ in which the directional derivative of $\psi$, in the direction of the vector $\vec{v}$, is equal to $P(r,s,\psi)$. This interpretation leads us to a method to solve (\ref{transpeq}).
	
	Curves $(r,X(r))$ in the $(r,s)$ plane that are tangent to $\vec{v}$ are called \emph{characteristic curves}. It follows that the characteristic curve passing through $(r,s)=(r_0,c)$ is the graph of the function $X$ that satisfies 
	\begin{equation}\label{EDO2}
	\frac{dX}{dr}=v(r,X(r))\quad\mbox{in which}\quad X(r_0)=c.
	\end{equation}

	Defining the value of ${\psi}$ along the characteristic curve $(r,X(r))$  by ${\Psi}(r):={\psi}(r,X(r))$, we have \begin{equation}\label{Psilinha}
	\frac{d}{dr}{\Psi}={\psi}_r + v{\psi}_s.
	\end{equation}
	Thus, the value of ${\psi}$ along a characteristic curve is determined by 
	\begin{equation}\label{EDO3}
	{\Psi}^{\prime}=P(r,X(r),\Psi(r)).
	\end{equation}
	
	The solution of \eqref{EDO3}, with initial data ${\Psi}(r_0)={\psi}_{r_0}(c)$, determines the value of ${\psi}$ along the characteristic curve that intersects the $(r_0,s)$-axis at $(r_0,c)$, because ${\Psi}(r_0)={\psi}(r_0,X(0))={\psi}(r_0,c)$. The surface $z=\psi(r,s)$ is the collection (or envelope) of space curves created as $c$ takes all real values.
	
	Next we demonstrate that this method is needed in order to solve an equation of the form 
	\eqref{transpeq} which will be useful to prove Theorem \ref{theonot}, Theorem \ref{teo1} and Theorem \ref{theo3}.
	
	\
	
	\noindent {\bf Proof of Theorem \ref{theonot}:} First of all, we derive the equation \eqref{R_2=0} to obtain  
	\[s\left[(r^2-s^2)^{3/2}(Q_s-sQ_{ss})\right]_r + (1-(r^2-s^2)2Q)r\left[(r^2-s^2)^{3/2}(Q_s-sQ_{ss})\right]_s=0.\]
	Note that a solution of these equation is 
	\begin{align}\label{QssQss}
	(r^2-s^2)^{3/2}(Q_s-sQ_{ss})=k,
	\end{align}
	where $k\in\mathbb{R}$. Hence, solving \eqref{QssQss} we get
	\[Q(r,s)=\frac{k\sqrt{r^2-s^2}s}{r^4} + \frac{c_1(r)}{2}s^2 + c_2(r),\]
	where $c_1(r)$ and $c_2(r)$ are smooth functions.
	
	Then, replacing $Q(r,\,s)$ in \eqref{R_2=0} we have 
	\[Q(r,s)=k\frac{\sqrt{r^2-s^2}s}{r^4} + \left(\frac{c_2'(r)+2rc_2^2(r)}{r(1-2r^2c_2(r))}-\frac{k^2}{2r^4(1-2r^2c_2(r))}\right)s^2 + c_2(r).\]
	Furthermore, if $k=0$ then $Q(r,\,s)$ defines a Douglas metric.
	
	In what follows, we consider $ c_2(r)=0, $  $ k\neq 0, $ $ r^2-s^2\neq 0 $ and $ s\neq 0$. Turns out that applying $Q(r,\,s)$ in \eqref{defQ} we need to solve 
	\begin{align}\label{eqtransp2}
	\left[1-\frac{(r^2-s^2)}{r^4}((r^2-s^2)-(ks-\sqrt{r^2-s^2})^2)\right]r{\psi}_s(r,s) + s{\psi}_r(r,s)=0,
	\end{align}
	in which \begin{align}\label{phissphiss} \psi(r,s)=\sqrt{r^2-s^2}(\phi(r,s)-s\phi_s(r,s)). \end{align}
	
	On the other hand, (\ref{eqtransp2}) is equivalent to a transport equation \[s\psi_r + v(r,s)\psi_s=0\]
	\[\psi(r_0,s)=\psi_{r_0}(s),\]
	where 
	\[v(r,s)= \frac{r}{s}\left[1-\frac{(r^2-s^2)}{r^4}\left((r^2-s^2)-\left(ks-\sqrt{r^2-s^2}\right)^2\right)\right].\]
	
	Considering the characteristic curves $ (r,X(r)) $ for the equation above, from \eqref{EDO2} we get
	\begin{align}\label{Xr}
	X'=\frac{r}{X}\left[1-\frac{(r^2-X^2)}{r^4}\left(r^2-X^2-(kX-\sqrt{r^2-X^2})^2\right)\right]    
	\end{align} which is equivalent to 
	\[\frac{r^3(r-XX')}{(r^2-X^2)^2}=1-\left(\frac{kX}{\sqrt{r^2-X^2}}-1\right)^2.\]
	By defining \[ \kappa(r)=k\frac{X(r)}{\sqrt{r^2-X(r)^2}}-1, \] we rewrite  \eqref{Xr} in the following form
	\[\frac{(1+k^2)(\kappa+1)\kappa'}{(\kappa+\frac{1}{1+k^2})^2 + \frac{k^2}{(1+k^2)^2}} = \frac{1}{r}.\]
	
	Integrating both therms we obtain the solution  $X(r)$ (implicitly) of \eqref{EDO2}:
	\[\frac{1}{2}\ln\left((1+k^2)\kappa^2+2\kappa+1\right) + k\arctan \left(\frac{(1+k^2)\kappa + 1}{k}\right) =(1+k^2)\ln(r) + k_1,\]
	where $k_1$ is obtained from the initial condition of \eqref{EDO2}.
	The characteristic curve through a given point $(r,s)$ passing through $(r_0,s)$ exists at $(r_0,c)$ in which
	\begin{eqnarray*}
		k_1(c)&=&\frac{1}{2}\ln\left((1+k^2)\kappa^2(r,s)+2\kappa(r,s)+1\right)\nonumber\\
		&&+ k\arctan \left(\frac{(1+k^2)\kappa(r,s) + 1}{k}\right) -(1+k^2)\ln(r).
	\end{eqnarray*}
	Straightaway equation \eqref{EDO3}, with initial condition $\Psi(r_0)=\psi(c)$, gives us $\Psi(r)=\psi(c).$ Then the solution of the initial problem \eqref{EDO2}-\eqref{EDO3} is
	\[\Psi(r,s)=\psi\left(k_{1}(c)\right).\]
	
	Therefore, we can infer that \eqref{varphi2} is a solution of \eqref{eqtransp2}. Thus, any differentiable real function $ \eta $ of (\ref{varphi2}) is the
	general solution of \eqref{eqtransp2}. It follows from (\ref{phissphiss}) that
	\begin{align*}
	\phi-s\phi_s=\frac{\eta(\overline{\varphi})}{\sqrt{r^2-s^2}}.
	\end{align*}
	Finally, we conclude that \eqref{phinondouglas} is the general solution of \eqref{eqtransp2}-\eqref{phissphiss}.
	\hfill $\Box$
	
	\vspace{.2in}
	
	In this part, we prove Theorem \ref{teo1} and Theorem \ref{theo3} making use of  Proposition \ref{propEdp}. 

	\
	
	\noindent {\bf Proof of Theorem \ref{teo1}:}
	We have already established that a Finsler metric $F=|y|\phi(r,s)$ is Douglas type if, and only if, there exist $f(r)$ and $g(r)$ smooth functions of $r$ such that $Q(r)=g(r)+\frac{s^2}{2}f(r)$. Thus, replacing this in the equation \eqref{R_2=0} we obtain $g(r)\neq \frac{1}{2r^2}$ and \[f(r)=\frac{2g'(r)+4rg^2(r)}{r-2r^3g(r)}.\] So, $\phi$ must satisfy the ordinary differential equation
	\begin{align}\label{eqdouglas}
	\frac
	1{2r}\frac{r\phi_{ss}-\phi_r+s\phi_{rs}}{\phi-s\phi_s+(r^2-s^2)\phi_{ss}}=g(r)+s^2\frac{g'(r)+2rg^2(r)}{r-2r^3g(r)}.
	\end{align}
	Now, for $r^2-s^2> 0$ and $s\neq 0$, \eqref{eqdouglas} is equivalent to \eqref{eqtransp}, where $\psi(r,s)=\sqrt{r^2-s^2}(\phi-s \phi_s)$ and $f(r)=\frac{2g'(r)+4rg^2(r)}{r-2r^3g(r)}$. Then, from \eqref{eqdouglas} we get
	\begin{align}\label{phi_sphi}
	\sqrt{r^2-s^2}(\phi-s \phi_s)= -s^2\sqrt{r^2-s^2}\left(\frac{\phi}{s}\right)_s=\eta(\varphi(r,s)),
	\end{align}
	where 
	\begin{align*}
	\varphi(r,s)=\frac{r^2-s^2}{4(r^2-s^2)\displaystyle\int \frac{g'(r)+2rg^2(r)}{(1-2r^2g(r))}T(r)dr - T(r)}.
	\end{align*}
	Here $\eta$ is any smooth real function of  $\varphi(r,s)$ and \[T(r):=\frac{e^{-\int\frac{4rg(r)}{1-2r^2g(r)}dr}}{(1-2r^2g(r))^2}.\]
	
	We acknowledge that Yu and Zhu \cite{YZ} gave a necessary and sufficient condition  for  $F=\alpha\phi(||\beta_{x}||_{\alpha},\frac{\beta}{\alpha})$ to be a Finsler metric for any $\alpha$ and $\beta$ with $||\beta_{x}||_{\alpha}<b_0$, where $\alpha$ and $\beta$ are, respectively, a Riemannian metric and a $1$-form. In particular,  considering $F(x,y)=|y|\phi(|x|,\frac{<x,y>}{|y|})$, then $F$ is a Finsler metric if, and only if, the positive function $\phi$ satisfies 
	\[
	\phi(s)-s\phi_s(s)+(r^2-s^2)\phi_{ss}(s)>0, \qquad \mbox{ when  } n\geq 2,
	\]
	such that  
	\[
	\phi(s)-s\phi_s(s)>0, \qquad  \mbox{ when } n\geq 3.
	\]
	
	Thus, for $F=|y|\phi(r,s)$, when $\phi$ is given by (\ref{phiscalar}), using \eqref{phi_sphi} we get
	\[\frac{-s}{\sqrt{r^2-s^2}}(\phi-s\phi) + \sqrt{r^2-s^2}(-s\phi_{ss})=\eta_s.\] 
	Therefore, $F$ defines a Finsler metric if, and only if, the inequalities (\ref{etateo2}) hold.
	\hfill $\Box$

	\

	\noindent {\bf Proof of Theorem \ref{theo3}}:
	At this point we look for spherically symmetric Finsler metric with constant flag curvature using Theorem 5.
	
	Similarly to the proof of Theorem 2, we  know that a Finsler metric $F=|y|\phi(r,s)$ is of Douglas type if, and only if, $Q(r)=g(r)+\frac{s^2}{2}f(r)$. Thus, replacing this in the equation \eqref{R_2=0} we obtain $g(r)\neq \frac{1}{2r^2}$ and \[f(r)=\frac{2g'(r)+4rg^2(r)}{r-2r^3g(r)}.\]
	Writing down $R_3=0$, given by \eqref{R2}, like $P:=P(r,\,s)$. Then, for  $r^2-s^2>0$ and $s\neq 0$, $R_3=0$ is equivalent to the transport equation  \eqref{eqtransp}, where $\psi(r,s)=\sqrt{r^2-s^2}(P-s P_s)$ and $f(r)=\frac{2g'(r)+4rg^2(r)}{r-2r^3g(r)}$. By Proposition \ref{propEdp} we have,
	\begin{align*}
	P(r,s)=s \left( h(r) - \int{\frac{\overline{\eta}(\varphi(r,s))}{s^2\sqrt{r^2-s^2}} ds}\bigr.\right) 
	\end{align*}
	such that
	\begin{equation*}
	\varphi(r,s)=\frac{r^2-s^2}{(r^2-s^2)\overline{T}(r) - T(r)},
	\end{equation*} 
	where $h$ and $\overline{\eta}$ are arbitrary smooth real functions.
	\noindent

	On the other hand  (see proof of Theorem 1.1 in \cite{2012arXiv1202.4543M}), we define
	\[U:=\frac{s\phi+(r^2-s^2)\phi_s}{\phi}, \quad \quad W:=\frac{s\phi_r+r\phi_s}{\phi}.\]
	From $U$ and $W$, we get
	\begin{align}\label{4.6}
	\phi_s=\frac{U-s}{r^2-s^2}\phi, \quad \quad \phi_r=\frac{1}{s}\left(W-\frac{r(U-s)}{r^2-s^2}\right)\phi.
	\end{align}
	Substituting $\phi_r$ and $\phi_s$ into \eqref{defQ} and \eqref{defP} we have
	\begin{align*}
	P=-QU + \frac{W}{2r},
	\end{align*}
	and
	\begin{align*}
	Q=\frac{1}{2rs}\frac{2rU-2rs-2r^2W+s(r^2-s^2)W_s+s^2W+sUW}{U^2-sU+(r^2-s^2)U_s}.
	\end{align*}
	The two identities above give us
	\begin{align*}
	U&=\frac{(r^2-s^2)(sP_s-2P)-s(1+sP)}{(r^2-s^2)(2Q-sQ_s)-sP-1},\\\\
	W&=2r\left(P+\frac{(r^2-s^2)(sP_s-2P)-s(1+sP)}{(r^2-s^2)(2Q-sQ_s)-sP-1}Q\right).
	\end{align*}
	
	
	From \eqref{4.6}, one can see that $\phi$ should satisfy
	\begin{eqnarray}\label{4.9}\displaystyle
	\begin{cases} \displaystyle
	(\ln \phi)_r = \frac{1}{s}\left(W-\frac{r(U-s)}{r^2-s^2}\right),  \\ \\ \displaystyle
	(\ln \phi)_s=\frac{U-s}{r^2-s^2}.
	\end{cases}
	\end{eqnarray}
	From above, we have
	\begin{align*}
	(\ln \phi)_{rs}=&-\frac{1}{s^2}\left(2r(P+UQ)-\frac{r(U-s)}{r^2-s^2}\right) + \frac{1}{s}\left(2r(P_s+U_sQ+UQ_s)\right)\\ &-r\left(\frac{U_s-1}{r^2-s^2}+\frac{2s(U-s)}{(r^2-s^2)^2}\right)
	\end{align*}
	and
	\begin{align*}
	(\ln \phi)_{sr}=\frac{U_r}{r^2-s^2}-\frac{2r(U-s)}{(r^2-s^2)^2}.
	\end{align*}
	
	The system \eqref{4.9} has a solution if, and only if, $(\ln \phi)_{rs}=(\ln \phi)_{sr}.$ A straightforward computation of this gives us 
	\[s^2U_r=rs(2(r^2-s^2)Q-1)U_s +(r+2r(r^2-s^2)(sQ_s-Q))U + 2r(r^2-s^2)(sP_s-P),\] which is the condition (\ref{conditionPQ}).
	
	\hfill $\Box$
	
	\section{New families of Douglas and projectively flat metrics with scalar and constant flag curvatures}\label{section4}
	
	\

	In this section, we provide new families of spherically symmetric Douglas metrics based on Theorem \ref{teo1} and Theorem \ref{theo3}.
	
	\begin{example}\label{ex001}
		Considering $g=0$, $\eta(x)=
		\left(\frac{\sqrt{x}}{(1-x)^{\frac{3}{2}}}\right)^3 + \epsilon \frac{\sqrt{x}}{(1-x)^{\frac{3}{2}}}$ and $\eta=\eta(r^2-s^2)$ in Theorem \ref{teo1} we have 
		\begin{align*}\phi(r,s)=&\frac{35r^{10}-245r^8s^2+490r^6s^4-392r^4s^6+112r^2s^8-140r^8+700r^6s^2-910r^4s^4+336r^2s^6}{(1-(r^2-s^2))^{\frac{7}{2}}(1-r^2)(1+r^4-2r^2)}\\ &+\frac{16s^8+210r^6-630r^4s^2+350r^2s^4+56s^6-140r^4+140r^2s^2+70s^4+35r^2+35s^2}{(1-(r^2-s^2))^{\frac{7}{2}}(1-r^2)(1+r^4-2r^2)} \\
		&+\epsilon\frac{1-r^2+2s^2}{(1-r^2)\sqrt{1-(r^2-s^2)}}.
		\end{align*}
		So, the Finsler metric \[F(x,y)=|y|\phi(|x|,\frac{<x,y>}{|y|})\] defined on $B(0,1)\subset\mathbb{R}^{n}$ is a projectively flat Finsler metric with scalar flag curvature.
	\end{example}

	The next result presents a family of examples for Theorem \ref{teo1}.
	
	\begin{corollary}\label{examD1}
		Consider
		\begin{align*}
		\phi_m(r,s)=&sh(r)-\gamma\sum_{i=0}^{m} (-1)^i\binom{m}{i} r^{2i} e^{r^2} s \int s^{2(m-i-1)}e^{-s^2} ds\\
		&+\epsilon e^{r^2-s^2}(\sqrt{\pi}se^{-s^2}\operatorname{erf}(s) + 1), &  m=0,1,2,\ldots
		\end{align*}
		where $\gamma\geq 0$ and $\epsilon>0$ are real constants,  $h(r)$ is any differential function such that $\phi_m(r,s)$ is positive, $\operatorname{erf}(x)=\frac{2}{\sqrt{\pi}}\int_{0}^xe^{-t^2}dt$ is the Gauss error function (non-elementary).
		Then, for any $m\in\mathbb{N}$ the following Finsler metric on $M_s^n$ 
		\[ F(x,y)=|y|\phi_m\left(|x|,\frac{<x,y>}{|y|}\right)\]
		is a projectively flat Finsler metric with scalar flag curvature. Moreover, $F$ is a projectively flat Finsler metric.
	\end{corollary}
	\begin{proof}
		Considering $g=0$, $\eta(r^2-s^2)=\sqrt{r^2-s^2}[(r^2-s^2)^m+\epsilon]e^{r^2-s^2}$, $\epsilon >0$. Then, the result follows from Theorem \ref{teo1} and a binomial expansion of $(r^2-s^2)^m$. In fact, the integral of the right-hand side of $\phi_{m}(r,\,s)$ may be expressed by 
		\begin{align*}
		\int s^{2(m-1)}e^{-s^2}ds&=-\frac{s^{2m-3}}{2}e^{-s^2}+ \frac{2m-3}{2}\int s^{2m-4}e^{-s^2}ds   &  m= 1,2 \ldots, 
		\end{align*}
		and
		\begin{align*}
		\int \frac{e^{-s^2}}{s^2}ds &= -\frac{e^{-s^2}}{s}[\sqrt{\pi}se^{s^2}(\operatorname{erf}(s))+ 1] + c& m=0, 
		\end{align*}
		where $\displaystyle\operatorname{erf}(x)=\frac{1}{\sqrt{\pi}}\int_{-x}^xe^{-t^2}dt$ is the Gauss error function (non-elementary)
		and $c\in \mathbb{R}$. 
	\end{proof}
	
	\begin{example}
		Taking $m=1$ in Corollary \ref{examD1} we have
		\[\phi(r,s)=sh(r)+e^{r^2-s^2}[\sqrt{\pi}se^{s^2}(r^2+\epsilon+\frac{1}{2})\operatorname{erf}(s)+r^2+\epsilon].\]
		Then, the following Finsler metric  \begin{align*}
		F(x,y)=|y|\phi\left(|x|,\frac{<x,y>}{|y|}\right)
		\end{align*} is a projectively flat Finsler metric with scalar flag curvature on $M_s^n.$  
	\end{example}

	\begin{example}
		
		For $m=2$ in Corollary \ref{examD1} we obtain
		\[\phi(r,s)=sh(r)+e^{r^2-s^2}[\sqrt{\pi}se^{s^2}(r^4+r^2+\epsilon-\frac{1}{4})\operatorname{erf}(s)+r^4+\frac{1}{2}s^2+\epsilon].\]
		Then, the following Finsler metric  \begin{align*}
		F(x,y)=|y|\phi\left(|x|,\frac{<x,y>}{|y|}\right)
		\end{align*} is a projectively flat Finsler metric with scalar flag curvature on $M_s^n.$  
	\end{example}

	Considering $g(r)=0$, $\eta(x)=\sqrt{x}(\gamma x^n\tanh^{-1}(x)+\epsilon)$ in Theorem \ref{teo1} from 
	\begin{align}
	\int s^{2(m-1)}\tanh^{-1}(r^2-s^2)ds=&-\frac{s^{2m-3}}{2}e^{-s^2}\nonumber\\
	&+ \frac{2m-3}{2}\int s^{2m-4}e^{-s^2}ds,&\mbox{for}\quad  m= 1,2 \ldots, \label{rec3}
	\end{align}
	and
	\begin{align}
	\int \frac{\tanh^{-1}(r^2-s^2)}{s^2}ds =& -\frac{\tanh^{-1}(r^2-s^2)}{s} + \frac{\tanh^{-1}\left(\frac{s}{\sqrt{r^2-1}}\right)}{\sqrt{r^2-1}}+\label{rec4}\\
	\nonumber    &-\frac{\tanh^{-1}\left(\frac{s}{\sqrt{r^2+1}}\right)}{\sqrt{r^2+1}}+ c& \mbox{for}\quad m=0, 
	\end{align}
	where $r>1$ and $c$ is constant, we get the corollary below. 
	
	\begin{corollary}\label{examD2}
		Let $\phi_m(r,s),$ $m\in\mathbb{N}$, be a function defined by
		\begin{align*}
		\phi_m(r,s)&=sh(r)-\gamma s\sum_{i=0}^{m}(-1)^i  \binom{m}{i} r^{2i} \int s^{2(m-i-1)}\tanh^{-1}(r^2-s^2)ds + \epsilon, 
		\end{align*}
		where $\gamma\geq 0$ and $\epsilon>0$ are real constants,  $h(r)$ is any differential function such that $\phi_{m}(r,s)$ is positive and the integral of the right-hand side is given by \eqref{rec3} and \eqref{rec4}. Then, the following Finsler metric   
		\[ F(x,y)=|y|\phi\left(|x|,\frac{<x,y>}{|y|}\right)\]
		defined on a symmetric space without the unit closed ball centered at the origin, is a Douglas metric with scalar flag curvature. Moreover, $F$ is a projectively flat Finsler metric.
	\end{corollary}

	\begin{example}
		For $m=0$ in Corollary \ref{examD2} we have
		\[\phi(r,s)=sh(r)+\gamma{\tanh^{-1}(r^2-s^2)} - \gamma s\frac{\tanh^{-1}\left(\frac{s}{\sqrt{r^2-1}}\right)}{\sqrt{r^2-1}}+\gamma s\frac{\tanh^{-1}\left(\frac{s}{\sqrt{r^2+1}}\right)}{\sqrt{r^2+1}} + \epsilon.\]
		Then, the following Finsler metric  \begin{align*}
		F(x,y)=|y|\phi\left(|x|,\frac{<x,y>}{|y|}\right)
		\end{align*} is a Douglas metric with scalar flag curvature. Furthermore, $F$ is a projectively flat Finsler metric.
	\end{example}
	
	From now on, we will consider $g(r)\neq 0.$ Inspired by Example \ref{ex001}, in Theorem \ref{teo1} we consider \[\varphi(r,s)=\frac{(r^2-s^2)}{T(r)-\overline{T}(r) (r^2-s^2)}\] and \[\eta(x)=\frac{\sqrt{x}}{(1-x)^{\frac{3}{2}}}\] 
	to determine the following family of examples.
	\begin{corollary}
		Let $g(r)$ be a smooth function such that $T(r)$ and $\overline{T}(r)$ in (\ref{T}) are well defined and   \[\phi(r,s)=sh(r)+\frac{(T(r)-r^2\overline{T}(r))(T(r)-(r^2-s^2)\overline{T}(r)) + s^2T(r)}{\sqrt{T(r)-(r^2-s^2)\overline{T}(r)}(T-r^2(\overline{T}(r)+1))^2}, \]
		where $h(r)$ is such that $\phi(r,s)$ is positive. Then, the following Finsler metric 
		\[ F(x,y)=|y|\phi\left(|x|,\frac{<x,y>}{|y|}\right)\]
		is a Douglas metric with scalar flag curvature. 
		
		\begin{remark}
			In the last corollary, if $T(r)=1$, $\overline{T}(r)=0$ and $h(r)=\pm\frac{ 2}{(1-r^2)^2}$, we obtain the projectively flat Berwald metric with vanishing flag curvature which was provided in \cite{berwald1929}: \[F(x,\,y)=\frac{(\sqrt{(1-|x|^2)|y|^2 +<x,y>^2} \pm <x,y>)^2}{(1-|x|^2)^2\sqrt{(1-|x|^2)|y|^2 +<x,y>^2}}.\]
		\end{remark}
	\end{corollary}
	Now, considering in Theorem \ref{teo1} \begin{align}\label{etaex}
	\eta(x)=\sqrt{-x}[(-x)^m+\epsilon],
	\end{align}
	for $\epsilon>0$ we have the next two corollaries considering $m=1$ and $m=2.$
	\begin{corollary}\label{corn1}
		Let $g(r)$ be a smooth function such that $T(r)$ and $\overline{T}(r)$ in (\ref{T}) are well defined and
		\begin{align*}
		\phi(r,s)=&sh(r) +\displaystyle\gamma\frac{r^2\sqrt{T(r)-\overline{T}(r)(r^2-s^2)}}{(T(r)-r^2\overline{T}(r))^2}  + \gamma s^2\frac{T(r)}{(T(r)-r^2\overline{T}(r))^2\sqrt{T(r)-\overline{T}(r)(r^2-s^2)}}
		\\
		&+\epsilon\frac{\sqrt{T(r)-\overline{T}(r)(r^2-s^2)}}{T(r)-r^2\overline{T}(r)}.
		\end{align*}
		Then, the following Finsler metric  
		\[ F(x,y)=|y|\phi\left(|x|,\frac{<x,y>}{|y|}\right)\]
		is a Douglas metric with scalar flag curvature.
	\end{corollary}
	\begin{example}
		Considering $g(r)=\frac{1}{2}$ and $h(r)=0$ in Corollary \ref{corn1}, we obtain  $T(r)=\overline{T}(r)=\frac{1}{1-r^2},$   and then 
		\[F(x,y)= \gamma\frac{(1-|x|^2)(|x|^2|y|^2+<x,y>^2) + 2<x,y>^2}{\sqrt{1-|x|^2}\sqrt{(1-|x|^2)|y|^2-<x,y>^2}} +\epsilon\frac{\sqrt{(1-|x|^2)|y|^2+<x,y>^2}}{\sqrt{1-|x|^2}}  \]
		defined on $TB(0,1)$ is a non-projectively flat Douglas metric with scalar flag curvature.
	\end{example}
	
	\begin{corollary}\label{corn2}
		Let $g(r)$ be a smooth function such that $T(r)$ and $\overline{T}(r)$ are well defined, $\epsilon >0$ and 
		\begin{align*}
		\phi(r,s)=&sh(r) +\frac{\left[(T(r)-r^2\overline{T}(r))(r^2-s^2)\right]^2 - \frac{4}{3} s^4T^2(r)}{(T(r)-r^2\overline{T}(r))^3(T(r)-(r^2-s^2)\overline{T}(r))} + \epsilon\frac{\sqrt{T(r)-\overline{T}(r)(r^2-s^2)}}{T(r)-r^2\overline{T}(r)}
		\end{align*}
		Then, the following Finsler metric  
		\[ F(x,y)=|y|\phi\left(|x|,\frac{<x,y>}{|y|}\right)\]
		is a Douglas metric with scalar flag curvature. 
	\end{corollary}

	Considering $m=0$ in \eqref{etaex} we obtain  \[\phi(r,s)=sh(r)+\frac{\sqrt{T(r)-\overline{T}(r)(r^2-s^2)}}{T(r)-r^2\overline{T}(r)}.\] The next two examples have this same structure.
	Here we have some examples considering $m=3,4, \ldots$ in (\ref{etaex}).

	\begin{example}\label{exs1}
		Considering $g(r)=-\frac{1}{r}$, $\eta(\varphi)=\sqrt{\varphi}$ in Theorem \ref{teo1}, we have \[T(r)=1, \quad \overline{T}(r)=-\frac{4}{r} \quad \text{ and } \quad \varphi(r,s)=\frac{r(r^2-s^2)}{r+4(r^2-s^2)}.\] 
		Assuming those identities in \eqref{phiscalar} we obtain the Finsler metric
		\[F:=h(|x|)<x,y> + c\frac{\sqrt{|x|^2|y|^2 + 4|x|(|x|^2|y|^2-<x,y>^2)}}{|x|(1+4|x|)},\]
		defined on $M_s^n\backslash\{0\}$ which is of Douglas type with scalar flag curvature.
	\end{example}

	\begin{example}\label{example02}
		Considering $h(r)=0$ in the last example we can prove that $R_3=0.$ Then,  the Finsler metric
		\[F:= c\frac{\sqrt{|x|^2|y|^2 + 4|x|(|x|^2|y|^2-<x,y>^2)}}{|x|(1+4|x|)}\]
		defined on $M_s^n\backslash\{0\}$ is of Douglas type with constant flag curvature $K=-\frac{4}{c^2}$.
	\end{example}
	
	\begin{example}
		Considering $h(r)=\frac{2c}{(1+2r)(1+4r)}$ in Example \ref{exs1} we can see that $R_3=0$. Therefore, the Finsler metric
		\[F:= c\frac{(2|x|+1)\sqrt{|x|^2|y|^2 + 4|x|(|x|^2|y|^2-<x,y>^2)}+2\sqrt{|x|}<x,y>}{|x|(1+2|x|)(1+4|x|)}\]
		defined on $M_s^n\backslash\{0\}$ is of Douglas type with constant flag curvature $K=-\frac{1}{c^2}$.
	\end{example}

	\begin{example}
		Considering $g(r)=-2$, $\eta(\varphi)=\sqrt{\varphi}$ in Theorem \ref{teo1}, we have \[T(r)=1, \quad \text{ and } \quad \varphi(r,s)=\frac{r(r^2-s^2)}{r+4(r^2-s^2)}.\] Assuming this in \eqref{phiscalar} we can infer that
		\[F:=h(r)<x,y> + c\sqrt{\frac{{1+4|x|^2}}{|y|^2+4(|x|^2-<x,y>^2)}}\]
		defined on $M_s^n\backslash\{0\}$ is a Finsler metric of Douglas type with scalar flag curvature.
	\end{example}

	\begin{example}\label{ex10}
		Applying the method described after Theorem \ref{theo3} give us the examples provided in \cite{2012arXiv1202.4543M}. In fact, from the same function $\phi$ given by Example \ref{exs1} we can rewrite (see Remark \ref{remark01}) $P:=\phi=h(r)s+\frac{c\sqrt{r(r+4(r^2-s^2))}}{r(1+4r)}$. Thus, from \eqref{conditionPQ} and \eqref{condU} we get the following examples:
		\[F:=|y|\frac{(2r+1)^2}{(4r+1)^{3/2}}\operatorname{exp}\left(\int_1^s \frac{\pm 4r(r+4r^2-2s^2)-4s(1+2r)\sqrt{r(r+4(r^2-s^2))}}{\left(r+4(r^2-s^2)\right)\left(\pm 2rs+(1+2r)\sqrt{r(r+4(r^2-s^2))}\right)}ds\right)\]
		and
		\[F=|y|\sqrt{\frac{1}{4r+1}\pm \frac{4\sqrt{r(r+4(r^2-s^2))}s}{r(2r+1)(4r+1)^2} - \frac{4(4r^2+3r+1)}{r(2r+1)^2(4r+1)^2}s^2}\]
		defined on $M_s\backslash\{0\}$.
	\end{example}

\end{document}